\documentclass[11pt]{amsart}
\usepackage{amsmath}
\usepackage{amssymb}
\usepackage{amsfonts}
\usepackage{graphicx}

\usepackage{xcolor}

\DeclareSymbolFont{AMSb}{U}{msb}{m}{n}
\DeclareMathSymbol{\I}{\mathbin}{AMSb}{"49}
\DeclareMathSymbol{\C}{\mathbin}{AMSb}{"43}

\newcommand{\U}{\mathcal{U}}
\newcommand{\N}{\mathbb{N}}
\newcommand{\Z}{\mathbb{Z}}
\newcommand{\Q}{\mathbb{Q}}
\newcommand{\R}{\mathbb{R}}
\newcommand{\V}{\mathcal{V}}

\newcommand{\F}{\mathcal{F}}

\newtheorem{theorem}{Theorem}[section]
\newtheorem{lemma}[theorem]{Lemma}
\newtheorem{proposition}[theorem]{Proposition}
\theoremstyle{definition}
\newtheorem{definition}[theorem]{Definition}
\newtheorem{remark}[theorem]{Remark}
\theoremstyle{corollary}
\newtheorem{corollary}[theorem]{Corollary}
\theoremstyle{example}

\theoremstyle{note}

\theoremstyle{notation}

\numberwithin{equation}{section}
\newtheorem*{theorem*}{Theorem}

\begin{document}
\title{Dynamical notions along filters}

\author{Lorenzo Luperi Baglini}
\address{Lorenzo Luperi Baglini, Dipartimento di Matematica, Università di Milano, Via Saldini 50, 20133 Milan, Italy. Supported by Grant P30821-N35 of the Austrian Science Fund FWF.}
\email{lorenzo.luperi@unimi.it}

\author{Sourav Kanti Patra}
\address{Sourav Kanti Patra, Department of Mathematics, Centre for Distance and Online Education, University of Burdwan, Burdwan-713104, West Bengal, India}
\email{souravkantipatra@gmail.com}

\author{Md Moid Shaikh}
\address{Md Moid Shaikh, Department of Mathematics, Maharaja Manindra Chandra College, 20 Ramkanto
Bose Street, Kolkata-700003, West Bengal, India}
\email{mdmoidshaikh@gmail.com}



\begin {abstract}
We study the localization along a filter of several dynamical notions. This generalizes and extends similar localizations that have been considered in the literature, e.g. near $0$ and near an idempotent. Definitions and basic properties of $\F$-syndetic, piecewise $\F$-syndetic, collectionwise $\F$-piecewise syndetic, $\F$-quasi central and $\F$-central sets and their relations with $\F$-uniformly recurrent points and ultrafilters are studied. We provide also the nonstandard characterizations of some of the above notions and we prove the partition regularity of several nonlinear equations along filters under mild general assumptions.

AMS subjclass [2020] : 37B20; 37B05; 05D10; 03H05.
\end{abstract}

\maketitle

\section{Introduction}


H. Furstenberg and B. Weiss first applied dynamical systems (topological dynamics) in Ramsey Theory in \cite{frus} and \cite{frusw}, starting the extremely fruitful use of ergodic methods in combinatorics, which has provided many fundamental results over the years. One of the reasons why these methods have been so succesful is because, in many cases, dynamical descriptions of Ramsey-theoretical problems are simpler than the algebraic or combinatorial ones; in other cases, the dynamical notions that arise from the study of combinatorial problems are interesting enough to be studied for themselves. Later, starting with the work of V. Bergelson and N. Hindman, several ergodic notions found an equivalent characterization in terms of special kinds of ultrafilters: for example, central sets where related to minimimal idempotents. In recent years, several papers have faced the problem of localizing known dynamical notions and results: for example, dynamical and combinatorial results near $0$ have been obtained in \cite{bayat, bhat, hindult, mine near 0, original near 0, sou}, and similar studies near an idempotent have been done in \cite{ms, mat}. The interplay between algebra and dynamics has been studied near zero in \cite{sou} and it has been extended to idempotents in \cite{ms}. Motivated by \cite{hindult} and \cite{mat}, the most generalized notion of largeness along a filter was introduced in \cite{oshu}.

In this present work, we want to explore a setting, studied also in \cite{pdeb} 
that unifies and extends all those that have been developed so far, namely the notion of dynamics along a filter. We show that many classical definitions and properties of notions of large sets can be extended to dynamics along a filter: properties of $\F$-syndetic sets and $\F$-uniformly recurrent points are studied in Section \ref{two}, $\F$-quasi central sets and their dynamics are studied in Section \ref{four}, collectionwise $\F$-piecewise syndetic sets and $\F$-central sets are  studied in Section \ref{five} and relations between $\F$-uniformly recurrent points and minimality are studied in Section \ref{six}. Finally, in Section \ref{seven} we provide some applications of our results to the partition regularity of nonlinear Diophantine equations along filters. Moreover, since nonstandard methods in combinatorics have become ever more used in recent years (see \cite{Gold}, where nonstandard characterizations of all classical notions we will consider here are provided, and \cite{advances} for combinatorial applications), we also provide the nonstandard characterizations of the basic dynamical notions along a filter in Section \ref{nschar}; this is the only section where a basic knowledge of nonstandard analysis is required.

\section{Basic results}\label{two}

We now present here some basic definitions, results,
and theorems related to topological dynamics that will be used frequently in this paper.
Throughout this paper,  $S$ is considered to be an arbitrary discrete semigroup (unless otherwise stated).
Now we state the well known definition of a dynamical system.
\begin{definition}\label{1.2} A dynamical system is a pair $(X,\langle T_s\rangle _{s \in S })$
such that
\begin{enumerate}
 \item $X$ is a compact Hausdorff space,
 \item $S$ is a semigroup,
\item for each $s\in S$, $T_s:X\rightarrow X$ and $T_s$ is continuous, and
\item for all $s,t\in S$ , $T_s\circ T_t=T_{st}$.
\end{enumerate}
\end{definition}

We recall the definitions of proximality from \cite[Definition 1.2(b)]{burns}, $U(x)$ from \cite[Definition 1.5(b)]{hindrec}
and uniform recurrence in a dynamical system  from \cite[Definition 1.2(c)]{burns}.

 \begin{definition}\label{1.4} Let $(X,\langle T_s\rangle _{s \in S})$ be a dynamical system.
\begin{enumerate}
 \item A point $y \in S$ is uniformly recurrent if and only if for every neighbourhood
    $U$ of $y$, $\{ s\in S : T_s(y) \in U  \}$ is syndetic.
 \item For $x\in X$, $U(x)=U_{X}(x)=\{p\in \beta S: T_p(x)$ is uniformly recurrent$\}$.
 \item The points $x$ and $y$ of $X$ are proximal if and only if for every neighbourhood
    $U$ of the diagonal in $X\times X$, there is some $s \in S$ such that
    $(T_s(x), T_s(y)) \in U$.
\end{enumerate}
    \end{definition}
 To give  the dynamical characterization of the members of certain idempotent
ultrafilters, we need the following definition  \cite[Definition 2.1]{john}.

    \begin{definition}\label{1.5} Let $S$ be a nonempty discrete space, let $K\subseteq S$ and let
$\mathcal{K}$ be a filter on $S$. We set
\begin{enumerate}
\item[(a)] $\overline{K}=\{p \in \beta S : K \in p \}$;
 \item [(b)] $\overline{\mathcal{K}}= \{p \in \beta S : \mathcal{K} \subseteq p \}$;
\item [(c)] $\mathcal{L}(\mathcal{K})=\{A\subseteq S: S\setminus A \not \in \mathcal{K}\}$.
\end{enumerate}
\end{definition}

Clearly, $\overline{\mathcal{K}}\subseteq \beta S$ and $\overline{\mathcal{K}}=\bigcap_{K\in \mathcal{K}}\overline{K}$. It is to be noted that $\overline{\mathcal{K}}\subseteq \beta S$ is closed and contains all the ultrafilters on $S$ that contain
$\mathcal{K}$. Conversely, every closed subset $C$ of $\beta S$ admits such a representation, as $$C=\overline{\{A\subseteq S\mid\ \forall p\in C A\in p\}}.$$ When
$\mathcal{K}$ is an idempotent filter, i.e., when $\mathcal{K}\cdot \mathcal{K}\subseteq\mathcal{K}$, $\overline{\mathcal{K}}$ is a semigroup; the converse is not always true.

In this paper, we focus on those filters that generate closed subsemigroups of $\beta S$.
Our main goal is to study properties of known dynamical notions when the dynamics is localized along a filter. Let us recall the following definitions 
from \cite{oshu}.
\begin{definition}
  Let $\mathcal{F}$ and $\mathcal{G}$ be two filters on $S$. A subset $A$ of $S$ is $(\mathcal{F}$, $\mathcal{G})$ -syndetic if for every $V\in \mathcal{F}$, there is a finite set $F \subseteq V$ such that $F^{-1}A \in \mathcal{G}$.
\end{definition}

\begin{definition}\label{2.7}  Let $T$ be a closed subsemigroup of  $\beta S$ and $\mathcal{F}$ be a filter on $S$ such that $\overline{\mathcal{F}}=T$.

\begin{enumerate}

 \item A subset $A$ of $S$ is $\mathcal{F}$-syndetic if for every $V\in \mathcal{F}$, there is a finite set
 $G \subseteq V$ such that $G^{-1}A \in \mathcal{F}$.
 \item A subset $A$ of $S$ is piecewise $\mathcal{F}$-syndetic if for every $V\in \mathcal{F}$, there is a finite
 $F_V\subseteq V$ and $W_V\in \mathcal{F}$ such that the family $$\{(x^{-1}F_V^{-1}A)\cap V: V\in \mathcal{F}, x\in W_V\}$$
 has the finite intersection property.
 \end{enumerate}
\end{definition}

In \cite{pdeb}, the analogous notion  of uniformly recurrent and proximality along a filter were introduced that will help to
discuss the dynamical characterizations of large sets along a filter.

 \begin{definition}\label{2.7} Let $(X,\langle T_s\rangle _{s\in S})$ be a dynamical system. Let $T$ be a closed subsemigroup of  $\beta S$ and $\mathcal{F}$ be a filter on $S$ such that $\overline{\mathcal{F}}=T$.
 \begin{enumerate}
 \item A point $x \in X$ is $\mathcal{F}$-uniformly recurrent if and only if for each neighbourhood $W$ of $x$, $\{s \in S: T_s(x) \in W  \}$ is $\mathcal{F}$-syndetic.
 \item Points $x$ and $y$ of $X$ are $\mathcal{F}$-proximal if and only if for every
neighbourhood $U$ of the diagonal in $X\times X$ and for each $V\in \mathcal{F}$ there
exists $s\in V$ such that $(T_s(x), T_s(y)) \in U$.
\end{enumerate}
\end{definition}

We now present some useful results from \cite{pdeb}, \cite{hindalg}, and \cite{oshu}. The first is \cite[Lemma 2.1]{oshu}.

\begin{lemma}\label{new 2.4}
 Let $T$ be a closed subsemigroup of $\beta S$, let $L$ be a minimal left ideal of $T$, let $\mathcal{F}$ and $\mathcal{G}$
 be the filters on $S$ such that $\overline{\mathcal{F}}=T$ and $\overline{\mathcal{G}}=L$, and $A\subseteq S$. Then the following statements are equivalent
 \begin{enumerate}
  \item $\overline A\cap L\neq \phi$,
  \item $A$ is $\mathcal{G}$-syndetic, and
  \item $A$ is $(\mathcal{F}$, $\mathcal{G})$ -syndetic
 \end{enumerate}

\end{lemma}

As a consequence, we have the following result, that generalizes \cite[Theorem 2.2]{oshu}, where only the equivalences between (a) and (b) below were shown.

\begin{theorem}\label{new 2.5}
 Let $T$ be a closed subsemigroup of $\beta S$, $\mathcal{F}$ be a filter on $S$ such that $\overline{\mathcal{F}}=T$,
 and $p\in T$. Then the following statements are equivalent:
 \item [(a)] $p\in K(T)$;
 \item [(b)] For all $A\in p$, $\{x\in S: x^{-1}A\in p\}$ is $\mathcal{F}$-syndetic;
 \item [(c)] For all $r\in T$, $p\in T\cdot r\cdot p$.
\end{theorem}
\begin{proof} That (a) implies (b) follows from the proof of \cite[Theorem 2.2]{oshu}.

To prove that (b) implies (c), let $r\in T$. For each $A\in p$, let $B(A)=\{x\in S: x^{-1}A\in p\}$ and
$C(A)=\{t\in S: t^{-1}B(A)\in r\}$. Observe that for any $A_1, A_2\in p$, one has $B(A_1\cap A_2)=B(A_1)\cap B(A_2)$ and
$C(A_1\cap A_2)=C(A_1)\cap C(A_2)$.

We claim that for every $A\in p$ and every $V\in \mathcal{F}$, $C(A)\cap V\neq \phi$. To see this, let $A\in p$ and
$V\in \mathcal{F}$ be given and pick $F\in \mathcal{P}_f(V)$ such that $F^{-1}B(A)\in \mathcal{F}$ so that
$F^{-1}B(A)\in r$. Hence there is some $t\in F$ with $t^{-1}B(A)\in r$. Then $t\in C(A)\cap V$. Thus $\{C(A)\cap V: A\in p$ and $V\in \mathcal{F}\}$ has the finite intersection property, so pick $q\in \beta S$ with $\{C(A)\cap V: A\in p$ and $V\in \mathcal{F}\}\}\subseteq q$. Then $q\in T$. We claim that $p=qrp$ for which it suffices (since both are ultrafilters) to show that $p\subseteq qrp$.

Let $A\in p$ be given. Then  $\{t\in S: t^{-1}B(A)\in r\}=C(A)\in q$ so $B(A)\in qr$ so $A\in qrp$ as required.

Finally, to prove that (c) implies (a) it sufficies to pick $r\in K(T)$.\end{proof}

The following theorem is based on \cite[Theorem 2.5]{bhat}.

\begin{theorem}\label{new 2.6}
 Let $T$ be a closed subsemigroup of $\beta S$, $\mathcal{F}$ be a filter on $S$ such that $\overline{\mathcal{F}}=T$ with the property that for each $V\in \mathcal{F}$, there exist $V_1, V_2\in \mathcal{F}$ satisfying $V_1V_2\subseteq V$, and $A\subseteq S$. Then the following statements are equivalent
 \begin{enumerate}

 \item[(a)] $A$ is piecewise $\mathcal{F}$-syndetic;
 \item[(b)] There exists $e\in E(K(T))$ such that $\{x\in S: x^{-1}A\in e\}$ is $\mathcal{F}$-syndetic;
 \item[(c)] There exists $e\in E(K(T))$ such that for every $V\in \mathcal{F}$ there exists $x\in V$ for which $x^{-1}A\in e$.
 \end{enumerate}
 \end{theorem}

\begin{proof} Let us prove that (1) implies (2). By \cite[Theorem  2.3]{oshu}, pick some $p\in K(T)$ with $A\in p$. Pick a minimal left ideal $L$ of $T$ with $p\in L$ and pick an idempotent $e\in L$. Then $p=pe$. We shall now show that $\{x\in S: x^{-1}A\in e\}$ is $\mathcal{F}$-syndetic. Let $V\in \mathcal{F}$ be given. Then there exists $V_1, V_2\in \mathcal{F}$ satisfying
$V_1V_2\subseteq V$. Since $A\in p=pe$, $\{y\in S: y^{-1}A\in e\}\in p$ and $V_1\in p$, we may pick $y\in V_1$ such that $y^{-1}A\in e$. Now by \cite{new 2.5}.(b) $B=\{z\in S:z^{-1}(y^{-1}A)\in e\}$ is $\mathcal{F}$-syndetic. So there
exists $F\in \mathcal{P}_f(V_2)$ such that $F^{-1}B\in \mathcal{F}$. Let $D=\{x\in S: x^{-1}A\in e\}$. We now claim that $\bigcup_{s\in G}s^{-1}D\in \mathcal{F}$, where $G=yF$. Note that $G\in \mathcal{P}_f(V)$, so our result is established provided the claim is proved. To this end, let $x\in V$ and pick $t\in F$ such that $tx\in B$. Then $(tx)^{-1}(y^{-1}A)\in e$, so $(ytx)^{-1}A\in e$. Thus $ytx\in D$. So $x\in (yt)^{-1}D$. Therefore $x\in s^{-1}D$ where $s=yt\in G$, as required.

That (2) implies (3) is trivial, as any $\F$-syndetic set has a non-empty intersection with sets in $\F$.

Finally, to prove that (3) implies (1) let $e\in E(K(T))$ be such that for every $V\in \mathcal{F}$ there exists $x\in V$ for which $x^{-1}A\in e$. Now choose $p\in T$ such that $E=\{x\in S: x^{-1}A\in e\}\in p$. Then $A\in pe$ and $pe\in K(T)$. So by \cite[Theorem  2.3]{oshu}, $A$ is piecewise $\mathcal{F}$-syndetic. \end{proof}

Let $X$ be a topological space and consider $X^X$ with the product topology.
Let $(X,\langle T_s \rangle_{s\in S})$ be a dynamical system; then, $\overline{\{T_s: s\in S\}}$ in $X^X$ is a semigroup
under the composition of mappings. The semigroup $\overline{\{T_s: s\in S\}}$ is the enveloping semigroup
of  $(X,\langle T_s\rangle _{s\in S})$. \cite[Theorem 19.11]{hindalg} shows a connection between a  dynamical system
$(X,\langle T_s\rangle _{s\in S})$ and $\beta S$ via enveloping semigroup of $(X,\langle T_s\rangle _{s\in S})$. 

\begin{theorem}\label{2.4}  Let $(X,\langle T_s \rangle_{s\in S})$ be a dynamical system and define
$\theta : S \rightarrow X^X$ by $\theta(s)=T_s$. Let $\tilde{\theta}$ be the continuous extension of $\theta$ to $\beta S$. Then $\tilde{\theta}$ is a
continuous homomorphism from $\beta S$ onto the enveloping semigroup of
$(X,\langle T_s\rangle _{s\in S})$.
\end{theorem}

We now state the Definition 19.12 from \cite{hindalg} 
 which is very useful.
\begin{definition}\label{2.5} Let $(X,\langle T_s\rangle _{s\in S})$ be a dynamical
system and define $\theta : S\rightarrow X^X$ by $\theta(s)=T_s$. For each
$p \in \beta S$, let $T_p= \tilde{\theta}(p)$.
\end{definition}
As an immediate consequence of Theorem \ref{2.4}, we have the following remark (see \cite[Remark 19.13]{hindalg}).

\begin{remark}\label{2.6}  Let $(X,\langle T_s\rangle _{s\in S})$ be a dynamical system and let
$p,q \in \beta S$. Then $T_p\circ T_q=T_{pq}$ and for each
$x \in X$, $T_p(x)$=$p$-$\lim_{s \in S}T_s(x)$.
\end{remark}

The relationships between $T_{p}$, for $p$ ultrafilter, and dynamical notions localized to $\F$ will be fundamental in the following. We recall the last three known results that we need for our studies. The first is \cite[Lemma 3.1]{pdeb}.
\begin{lemma}\label{2.8}
Let $(X,\langle T_s\rangle _{s\in S})$ be a dynamical system and let
$x, y\in X$. Then $x$ and $y$ are $\mathcal{F}$-proximal if and only if  there
is some $p \in \overline{\mathcal{F}}$ such that $T_p(x)=T_p(y)$.
\end{lemma}

The second, that shows the relationships between some algebraic and dynamical notions, is \cite[Theorem 3.2]{pdeb}.

\begin{lemma}\label{2.9}
Let $(X,\langle T_s\rangle _{s\in S})$ be a dynamical system and $L$ be a minimal left ideal of $\overline{\mathcal{F}}$ and $x \in X$.
The following statements are equivalent:
\begin{enumerate}
 \item The point $x$ is an $\mathcal{F}$-uniformly recurrent point  of $(X,\langle T_s\rangle _{s\in S})$.
  \item There exists $u \in L$ such that $T_u(x)=x$.
  \item  There exist $y \in X$ and an idempotent $u \in L$ such that $T_u(y)=x$.
  \item  There exists an idempotent $u \in L$ such that $T_u(x)=x$.
\end{enumerate}
\end{lemma}

The last is \cite[Lemma 3.3]{pdeb}.

\begin{lemma}\label{2.10}
Let $(X,\langle T_s\rangle _{s \in S})$ be a dynamical system and let $x \in X$.
Then for each $V \in \overline{\mathcal{F}}$ there is a ${\mathcal{F}}$-uniformly recurrent point
$y \in \overline{\{ T_s(x): s \in V\}}$ such that $x$ and $y$ are
${\mathcal{F}}$-proximal.
\end{lemma}

\section{Nonstandard characterizations}\label{nschar}

It is well known that the notions of syndetic and piecewise syndetic (and many other related ones) have simple characterization in nonstandard terms. Such characterization have been fundamental to develop many applications of such notions; we refer to \cite{Gold} and references therein. For this reason, we believe that is relevant to generalize these nonstandard characterizations and proofs to our present setting of dynamics along filters; to this end, we will follow the notations and, when possible, the proofs of \cite{Gold}. In the present setting there are some additional technical difficulties with respect to the classical case.

Solely in this section, we assume the reader to be familiar with the basics of nonstandard analysis. The nonstandard take on ultrafilters that we discuss here has been used in many recent papers to produce several results in combinatorial number theory (see e.g. \cite{advances, monads, integers}). We recall its basic definition and facts.

We work in a setting that allows for iterated nonstandard extensions, which we assume to be sufficiently saturated. For any $A\subseteq S$, we let inductively $A^{(0)\ast}:=A$ and\footnote{When $n=1,2$ we use the more common and simpler notations $A^{\ast}, A^{\ast\ast}$.} $A^{(n+1)\ast}:=\left(A^{(n)\ast} \right)^{\ast}$; we set $A^{(\infty)\ast}:=\bigcup_{n\in\N}A^{(n)\ast}$. We use the same kind on notations when the star map is applied to nonstandard objects: e.g., whenever $\alpha\in S^{(\infty)\ast}$ we let $\alpha^{(3)\ast}:=\alpha^{\ast\ast\ast}=\left(\left(\alpha^{\ast}\right)^{\ast}\right)^\ast$. 

For $\alpha\in S^{(\infty)\ast}$, we set 
\[\ell(\alpha):=\min\left\{n\in\N\mid \alpha\in S^{(n)\ast}\right\}.\]
Given $\F$ filter on $S$ we set 
\[\mu(\F):=\left\{\alpha\in S^{\ast}\mid \forall A\in\F \ \alpha\in A^{\ast}\right\}\]
and 
\[\mu_{\infty}(\F):=\left\{\alpha\in S^{(\infty)\ast}\mid \forall A\in\F \ \alpha\in A^{(\infty)\ast}\right\}.\]
$\mu_{\infty}(\F)$ will be called the monad of $\F$.

Conversely, given $\alpha\in S^{(\infty)\ast}$, we let 
\[\U_{\alpha}:=\left\{A\subseteq S\mid \alpha\in A^{(\infty)\ast}\}= \{A\subseteq S\mid \alpha\in A^{(\ell(\alpha))\ast}\right\}.\]
For $\alpha,\beta\in S^{(\infty)\ast}$, we say that $\alpha,\beta$ are $u$-equivalent (notation: $\alpha\sim\beta$) if $\U_{\alpha}=\U_{\beta}$. 

It is rather simple (see e.g. \cite{monads} for details and a detailed study of the properties of these nonstandard characterizations) to prove that, for all $\alpha\in S^{(\infty)\ast}$, $\U_{\alpha}\in\beta S$ and, conversely (assuming sufficient saturation), that $\mu(\F)$ is nonempty for all $\F$ filters on $S$; for $\U$ nonprincipal, $\mu(\U)$ will be infinite, its cardinality depending on that of $S^{\ast}$. Moreover, 
\[\alpha\in \mu(\F)\Leftrightarrow \F\subseteq\U_{\alpha},\]
namely $\mu(\F)=\bigcup_{\U\supseteq\F}\mu(\U)$.

In general, $\U_{\alpha}\cdot\U_{\beta}\neq\U_{\alpha\cdot\beta}$; however, in our nonstandard setting, we have that\footnote{In what follows and in most applications, one uses this formula with $\alpha,\beta\in S^{\ast}$; we wrote here the more general formulation as we will need to use it in two proofs.} $\forall\alpha,\beta\in S^{(\infty)\ast}$ 
\[\U_{\alpha}\cdot\U_{\beta}=\U_{\alpha\cdot\beta^{(\ell(\alpha))\ast}}.\]
For this reason, we say that $\alpha\in S^{(\infty)\ast}$ is idempotent if $\alpha\sim \alpha\cdot \alpha^{(\ell(\alpha))\ast}$, i.e. if $\U_{\alpha}$ is idempotent, and that $\alpha\in S^{\ast}$ is $\F$-minimal if $\U_{\alpha}\in\overline{\F}\cap K(\beta S)$. We say that $\alpha$ is an $\F$-minimal idempotent if it is $\F$-minimal and idempotent. Notice that if $\alpha$ is $\F$-minimal and $\beta,\gamma\in\F^{\ast}$ then also $\alpha\beta^{\ast}$, $\gamma\alpha^{\ast}$, $\gamma\alpha^{\ast}\beta^{\ast\ast}$ are minimal.

For $X\subseteq\beta S$, we set 
\[\mu(X):=\bigcup_{\U\in X}\mu(\U);\]
conversely, for $A\subseteq S^{(\infty)\ast}$ we let 
\[\pi (A):=\{\U\in\beta S\mid \exists \alpha\in A \ \U=\U_{\alpha}\}.\]
The nonstandard characterizations of $\F$-syndetic and piecewise $\F$-syndetic sets are given in the following proposition.

\begin{proposition}\label{nonst} Let $A\subseteq S$ and $\F$ filter on $S$. The following equivalences hold:
\begin{enumerate}
\item $A$ is $\F$-syndetic if and only if there exists $\Gamma\subseteq\mu(\F)$ hyperfinite with $\Gamma^{-1}A^{\ast}\in\F^{\ast}$;

\item $A$ is piecewise $\F$-syndetic if and only if there exists $W\in\F^{\ast},\Gamma\subseteq\mu(\F)$ hyperfinite and $\beta\in\mu(\F)$ such that $\beta^{\ast}\in\bigcap_{x\in W}x^{-1}\Gamma^{-1}A^{\ast\ast}$, i.e. such that $W\beta^{\ast}\subseteq\Gamma^{-1}A^{\ast\ast}$.

\end{enumerate}
\end{proposition}

\begin{proof} (1) Assume that $A$ is $\F$-syndetic. For all $V\in\F$ set 
\[I_{V}:=\{G\in\wp_{fin}\left(V \right)\mid G^{-1}A\in\F\}.\]
The family $\{I_{V}\}_{V\in\F}$ has the finite intersection property, as $\emptyset\neq I_{V_{1}\cap\dots\cap V_{n}}\subseteq I_{V_{1}}\cap\dots\cap I_{V_{n}}$ for all $n\in\N$ and $V_{1},\dots, V_{n}\in \F$. Let $\Gamma\in\bigcap_{V\in\F} V^{\ast}$. Then $\Gamma$ is hyperfinite and $\Gamma\subseteq V^{\ast}$ for all $V\in\F$ by construction, hence $\Gamma\subseteq\mu(\F)$. And, by transfer, as $\Gamma\in I_{V}^{\ast}$ it follows that $\Gamma^{-1}A^{\ast}\in\F^{\ast}$.

Conversely, let $\Gamma\subseteq\mu(\F)$ hyperfinite be such that $\Gamma^{-1}A^{\ast}\in\F^{\ast}$. In particular, for all $V\in\F$ $\Gamma\in \wp_{fin}(V)^{\ast}$, so (with the same notations used above) $I_{V}^{\ast}\neq\emptyset$. By transfer, $I_{V}\neq\emptyset$, so $A$ is $\F$-syndetic.

(2) $A$ is piecewise $\F$-syndetic if and only if for all $V\in\F$ there are $V_{F},W_{F}$ as in Definition \ref{new 2.6} such that the family 
\[\mathcal{G}:=\{(x^{-1}F_{V}^{-1}A)\cap V\mid V\in\F, x\in W_{V}\}\]
has the FIP. This is equivalent to say that there exists an ultrafilter $\U_{\beta}\in\beta S$ that extends $\mathcal{G}$ which, by definition of $\U_{\beta}$, is equivalent to say that there exists $\beta\in S^{\ast}$ such that $\forall V\in\F\ \forall x\in W_{V}\ \exists f\in F_{V} \ fx\beta\in A^{\ast}$. 
Now for all $V\in\F$ let
\[I_{V}(\beta)=\{\left(F_{V},W_{V} \right)\mid \forall x\in W_{V}\exists f\in F_{V} fx\beta\in A^{\ast}\}.\]
The family $\{I_{V}\}_{V\in\F}$ has the FIP, as $I_{V_{1}}(\beta)\cap\dots\cap I_{V_{n}}(\beta)\supseteq I_{V_{1}\cap\dots\cap V_{n}}(\beta)\neq\emptyset$. Pick $(\Gamma,W)\in\bigcap_{V\in\F} I_{V}(\beta)^{\ast}$. Then, by definition, $W\in\F^{\ast}, \Gamma\subseteq\mu(\F)$ is hyperfinite and $W\beta^{\ast}\subseteq\Gamma^{-1}A^{\ast\ast}$, as required.

Conversely, let $\beta,W,\Gamma$ as in the hypothesis be given. Then for all $V\in\F$ the following holds true:
\[\exists W\in\F^{\ast}\ \exists \Gamma\in\wp_{fin}(V)^{\ast} \ \beta^{\ast}\in\left(\bigcap_{x\in W}x^{-1}\Gamma^{-1}A^{\ast\ast}\right)\cap V^{\ast\ast}.\]
By transfer then, for all $V\in\F$ we have that
\[\exists W_{V}\in F\ \exists F_{V}\in\wp_{fin}(V) \ \beta\in \left(\bigcap_{x\in W_{V}}x^{-1}F_{V}^{-1}A^{\ast}\right)\cap V^{\ast}.\]
For all $V\in\F$ take $W_{V},F_{V}$ as above. As $\beta\in \left(\bigcap_{x\in W_{V}}x^{-1}F_{V}^{-1}A^{\ast}\right)\cap V^{\ast}$, in particular it means that $\U_{\beta}$ extends $\mathcal{G}:=\{(x^{-1}F_{V}^{-1}A)\cap V\mid V\in\F, x\in W_{V}\}$, therefore $\mathcal{G}$ has the FIP, which proves that $A$ is piecewise $\F$-syndetic. \end{proof}

Notice that, in the nonstandard characterization of piecewise $\F$-syndeticity, by using Lemma \ref{banale} below we could have additionally asked that $W\subseteq\mu(\F)$.

To generalize the nonstandard characterizations of piecewise syndetic sets in terms of minimal points and central sets to dynamics along a filter $\F$, we will need some useful results about generators of filters $\F$, some of which require $\overline{\F}$ to be a semigroup.

\begin{lemma}\label{banalissimo} Let $\F$ be a filter on $S$. Then for all $W\in\F^{\ast}$, for all $\U\in\overline{\F}$ $W\cap\mu(\U)\neq\emptyset$.\end{lemma}

\begin{proof} We just have to observe that $W\in\F^{\ast}\subseteq\U^{\ast}$, hence $I_{A}:=W\cap A^{\ast}\neq\emptyset$ for all $A\in\U$. As the family $\{I_{A}\}_{A\in\U}$ has the FIP, by saturation we deduce that $\bigcap_{A\in\U} W\cap A^{\ast}\neq\emptyset$, and any $\alpha$ in this intersection is, in particular, in the monad of $\U$. \end{proof}

\begin{lemma}\label{banale} Let $\F$ be a filter on $S$. Then there exists $W\in\F^{\ast}$ such that $W\subseteq\mu\left(\F\right)$.\end{lemma}

\begin{proof} For all $A\in\F$ let $I_{A}=\{B\in\F\mid B\subseteq A\}$. Clearly, $I_{A}\neq\emptyset$ and $\{I_{A}\}_{A\in\F}$ has the FIP, therefore by enlarging $\bigcap_{A\in\F} I_{A}^{\ast}\neq\emptyset$. It remains to notice that any $W$ in this intersection has the desired property.\end{proof}

\begin{theorem}\label{hope} Let $\F$ be a filter on $S$ and assume that $\overline{\F}$ is a semigroup. Let $\Gamma\subseteq\mu(\F)$ hyperfinite, $W\in\F^{\ast}$ and $\alpha\in\mu(\F)$. Let $L$ be the left $\F$-ideal generated by $\U_{\alpha}$. Then there exists $\tau \in\mu(\F)\cap W$ such that $\tau\sim\alpha$ and $\forall\gamma\in\Gamma \ \gamma\tau\in\mu(L)$. \end{theorem}

\begin{proof} Let $A\subseteq S,\gamma\in\Gamma$. By definition, 
\[A\in\U_{\gamma}\cdot\U_{\alpha}\Leftrightarrow\{s\in S\mid\{t\in S\mid st\in A\}\in\U_{\alpha}\}\in\U_{\gamma},\]
hence 
\[A\in\U_{\gamma}\cdot\U_{\alpha}\Leftrightarrow\gamma\in\{s\in S\mid\{t\in S\mid st\in A\}\in\U_{\alpha}\}^{\ast}.\]
As $\{s\in S\mid\{t\in S\mid st\in A\}\in\U_{\alpha}\}^{\ast}=\{\sigma\in S^{\ast}\mid\{\tau\in S^{\ast}\mid st\in A^{\ast}\}\in\U_{\alpha}^{\ast}\}$, this shows that 
\[A\in\U_{\gamma}\cdot\U_{\alpha}\Leftrightarrow I^{\gamma}_{A}:=\{\tau\in S^{\ast}\mid \gamma\tau\in A^{\ast}\}\in\U_{\alpha}^{\ast}.\]

In particular, by Lemma \ref{banale} pick $T\in\U_{\alpha}^{\ast}$ with $T\subseteq\mu\left(\U_{\alpha}\right)$. Let $\mathcal{G}$ be the filter on $S$ such that $\overline{\mathcal{G}}=L$. If $A\in L$ then, as $\U_{\gamma}\cdot\U_{\alpha}\in L$, we have that $A\in\U_{\gamma}\cdot\U_{\alpha}$. So for all $\gamma\in\Gamma$ and $A\in\mathcal{G}$, $I_{A}^{\gamma}\in\U_{\alpha}$. Therefore also $I_{A}^{\gamma}\cap T\in\U_{\alpha}$ and, as $\Gamma$ is hyperfinite, we have that $I_{A}=\left(\bigcap_{\gamma\in\Gamma} \{\tau\in S^{\ast}\mid \gamma\tau\in A^{\ast}\}\right)\cap T\in\U_{\alpha}$.
Notice that $I_{A}$ is an internal set for all $A\in\mathcal{G}$, and that the family $\{I_{A}\}_{A\in\mathcal{G}}$ has the FIP. Hence, by saturation, $\bigcap_{A\in\mathcal{G}} I_{A}\neq \emptyset$.

We claim that any $\tau$ in the above nonempty intersection satisfies the conclusions of our Theorem. In fact, $\alpha\sim\tau$ as $\tau\in T\subseteq\mu\left(\U_{\alpha}\right)$ and, for all $\gamma\in\Gamma$, by construction $\gamma\tau\in A^{\ast}$ for all $A\in\mathcal{G}$, namely $\gamma\tau\in\mu(L)$. \end{proof}

In the following, given $\alpha\in S^{\ast}$ and $A\subseteq S$, we let 
\[A_{\alpha}:=\{s\in S\mid s\cdot\alpha\in A^{\ast}\}.\] 

First, we provide a nonstandard proof of the nonstandard formulation of Theorem \ref{new 2.5}.

\begin{theorem}\label{nonstandard minimal} Let $\alpha\in S^{\ast}$, let $\F$ be a filter on $S$ and assume that $\overline{\F}$ is a semigroup. The following facts are equivalent:
\begin{enumerate}
\item\label{A} $\alpha$ is $\F$-minimal;
\item\label{B} $\forall A\in\U_{\alpha}\ A_{\alpha}$ is $\F$-syndetic;
\item\label{C} $\forall \beta\in\mu(\F)\ \exists\gamma\in\mu(\F)$ such that $\alpha\sim\gamma\cdot\beta^{\ast}\cdot\alpha^{\ast\ast}$.
\end{enumerate}
\end{theorem}

\begin{proof} $( \ref{A}) \Rightarrow ( \ref{B})$ Fix $A\in\U_{\alpha}$. As $\alpha$ is minimal, there is $L$ minimal left ideal in $\overline{\F}$ such that $\U_{\alpha}\in L$. Let $\beta\in\mu(L)$. As $L$ is minimal, there exists $\gamma\in\mu(\F)$ such that $\gamma\cdot\beta^{\ast}\sim\alpha$. In particular, for all $F\in\F$ we have 
\[\exists \gamma\in F^{\ast} \ \gamma\cdot\beta^{\ast}\in A^{\ast\ast}. \]
By transfer, it follows that $\forall F\in \F\ \exists f\in F \ f\cdot \beta\in A^{\ast}$. 

As the above is true for any $\beta\in\mu(L)$, it is in particular true for any object of the form $\delta\cdot\alpha^{\ast}$ with $\delta\in\mu(F)$. Thus, 
\[\forall\delta\in\mu(\F)\ \forall F\in\F \ \delta\cdot\alpha^{\ast}\in f^{-1}A^{\ast\ast}.\] 
Therefore,
\[\forall\delta\in\mu(\F)\ \forall F\in\F\ \exists f\in F\ \gamma\cdot\alpha^{\ast}\in f^{-1}A^{\ast\ast},\]
which means that
\[\gamma\in\{\eta\in S^{\ast}\mid f\eta\alpha^{\ast}\in A^{\ast\ast}\}=\{s\in S\mid fs\alpha\in A^{\ast}\}^{\ast}=\left(f^{-1}A_{\alpha} \right)^{\ast}.\]
As $\gamma\in\mu(\F)$, this shows that $f^{-1}A_{\alpha}\in\F$, hence that $A_{\alpha}$ is $\F$-syndetic.

$( \ref{B}) \Rightarrow ( \ref{C})$ Fix $\beta\in\mu(\F), A\in\U_{\alpha}$. By hypothesis, $A_{\alpha}$ is $\F$-syndetic, namely $\forall F\in\F \ \exists H\in\wp_{fin}(F) \ H^{-1}A_{\alpha}\in \F$. As $\beta\in\mu(\F)$, we have that $\beta\in \left(H^{-1}A \right)^{\ast}=H^{-1}\left(A^{\ast} \right)$, so $\beta\in f^{-1}A^{\ast}_{\alpha}$ for some $f\in H$. As, by transfer, $A_{\alpha}=\{s\in S\mid s\cdot\alpha\in A^{\ast}\}$, it follows that $\beta\in f^{-1}\left(A^{\ast}_{\alpha} \right)\Leftrightarrow f\beta\in A_{\alpha}^{\ast}\Leftrightarrow f\beta\alpha^{\ast}\in A^{\ast\ast}$. 

Now, for $F\in \F, A\in\U_{\alpha}$ let 
\[\Gamma_{F}^{A}=\{f\in F\mid f\beta\alpha^{\ast}\in A^{\ast\ast}\}.\]

The family $\{\Gamma_{F}^{A}\}_{F\in\F, A\in\U_{\alpha}}$ has the FIP as, for all $n\in\N$,
\[\Gamma_{F_{1}}^{A_{1}}\cap\dots\cap\Gamma_{F_{1}}^{A_{1}}\supseteq \Gamma_{F_{1}\cap\dots\cap F_{n}}^{A_{1}\cap\dots\cap A_{n}}\neq\emptyset.\]
By enlarging, $\bigcap_{F\in\F, A\in\U_{\alpha}} \left(\Gamma_{F}^{A} \right)^{\ast}\neq \emptyset$. If $\gamma$ belongs to this intersection, by construction $\gamma\in\mu(\F)$ is such that $\gamma\cdot\beta^{\ast}\cdot\alpha^{\ast\ast}\in\mu\left(\U_{\alpha} \right)$, i.e. $\gamma\cdot\beta^{\ast}\cdot\alpha^{\ast\ast}\sim\alpha$.

$( \ref{C}) \Rightarrow( \ref{A})$ Take $\beta\in\mu(\F)$ minimal, take $\gamma\in\mu(\F)$ such that $\alpha\sim \gamma\beta^{\ast}\alpha^{\ast\ast}$. We conclude as $\gamma\beta^{\ast}\alpha^{\ast\ast}$ is minimal. \end{proof}

As a consequence, we have the following:

\begin{theorem}\label{T1} Let $A\subseteq S$, let $\overline{F}$ be a filter on $S$ and let $\overline{F}$ be a semigroup. Then $A$ is piecewise $\F$-syndetic if and only if there exists an $\F$-minimal $\alpha\in A^{\ast}$. \end{theorem}

\begin{proof} $(1)\Rightarrow (2)$ Let $\beta\in\mu(\F), W\in\F^{\ast},\Gamma\subseteq\mu(\F)$ hyperfinite be such that $W\beta^{\ast}\subseteq\Gamma^{-1}A^{\ast\ast}$. By Lemma \ref{banalissimo} and Lemma \ref{banale}, we can assume that $W\subseteq\mu(\F)$ and $\forall \alpha\in\mu(\F)\ \exists \tau\in\mu(\F) \alpha\sim\tau$. Hence, in particular, there exists $\alpha\in W$ that is $\F$-minimal. Let $L$ be the left $\F$-ideal generated by $\U_{\alpha}$. By Theorem \ref{hope}, in particular there exists $\tau\in W$ such that $\forall\gamma\in\Gamma \ \gamma\tau\in\mu(L)$. By hypothesis, there exists $\gamma\in\Gamma$ such that $\gamma\tau\beta^{\ast}\in A^{\ast\ast}$. We just have to observe that $\gamma\tau\beta^{\ast\ast}$ is minimal since $\gamma\tau$ is minimal, $\U_{\gamma\tau}\cdot\U_{\beta}\in K(\F)$ and $\gamma\tau\beta^{\ast}\in\mu\left(\U_{\gamma\tau}\cdot\U_{\beta}\right)$.

$(2)\Rightarrow (1)$ Let $\alpha\in A^{\ast}$ minimal. By Theorem \ref{nonstandard minimal} it follows that $A_{\alpha}$ is $\F$-syndetic. Let $\Gamma\subseteq\mu(\F)$ hyperfinite be such that $W=\Gamma^{-1}A^{\ast}_{\alpha}\in\F^{\ast}$. We claim that 
\[\alpha^{\ast}\in\bigcap_{w\in W}w^{-1}\Gamma^{-1}A^{\ast\ast},\]
which would conclude by Proposition \ref{nonst}. By definition, $w\in W$ if and only if there exists $\gamma\in\Gamma, a\in A^{\ast}_{\alpha}$ such that $\gamma w=a\in A^{\ast}_{\alpha}$. As $A^{\ast}_{\alpha}=\{\eta\in S^{\ast}\mid \eta\cdot \alpha^{\ast}\in A^{\ast\ast}\}$, this shows that for all $w\in W$ there exists $\gamma\in\Gamma$ such that $\gamma w \alpha^{\ast}\in A^{\ast\ast}$, which proves our claim.\end{proof}

We can now prove a nonstandard version of Theorem \ref{new 2.5}.

\begin{proof}[Nonstandard Proof of Theorem \ref{new 2.5}.]

$(1)\Rightarrow (2)$ By Theorem \ref{T1}, there is $\alpha\in A^{\ast}$ $\F$-minimal. Let $L$ be the minimal left $\F$-ideal such that $\alpha\in L$. Let $\beta\in\mu(L)$ be a minimal idempotent. Then $A_{\beta}=\{s\in S\mid s\beta\in A^{\ast}\}\subseteq B$. If we prove that $A_{\beta}$ is $\F$-syndetic, $B$ is then $\F$-syndetic as well. As $\alpha,\beta\in \mu(L)$, there is $\gamma\in\mu(\F)$ such that $\alpha\sim\gamma\beta^{\ast}$. Then $\alpha\beta^{\ast}\sim\gamma\beta^{\ast}\beta^{\ast\ast}\sim\gamma\beta^{\ast}\sim\alpha$. As $\alpha\in A^{\ast}$, it follows that also $\alpha\beta^{\ast}\in A^{\ast}$, and this holds if and only if 
\[\alpha\in \{s\in S\mid s\beta\in A^{\ast}\}^{\ast}.\]
This shows that $A_{\beta}\in \U_{\alpha}$ and, as $\alpha$ is minimal, we conclude that $A_{\beta}$ is $\F$-syndetic by Theorem \ref{nonstandard minimal}. 

$(2)\Rightarrow (3)$ As $B$ is $\F$-syndetic, there exists $\Gamma\subseteq\mu(\F)$ hyperfinite with $\Gamma^{-1}B^{\ast}\in\F^{\ast}$. Given $\alpha\in\mu(\F)$, let $\beta\in\mu(\F)\cap\Gamma^{-1}B^{\ast}$ be given by Theorem \ref{hope}, namely $\beta\sim\alpha$ and $\gamma\cdot \beta\in \mu(L)$ for all $\gamma\in\Gamma$, where $L$ is the left $\F$-ideal generated by $\U_{\alpha}$. As $\beta\in\Gamma^{-1}B^{\ast}$, there exists $\gamma\in\Gamma$ such that $\gamma\beta\in B^{\ast}$. As $\gamma\cdot\beta\in L$  and $\overline{\F}$ is a semigroup, $\gamma\beta\in \mu(\F)$; in particular, $\gamma\beta\in V^{\ast}\cap B^{\ast}$. By transfer, we have that $B\cap V\neq \emptyset$.

$(3)\Rightarrow (1)$ For all $V\in\F$ let 
\[ I_{V}=\{\U\in\beta S\mid \U\in \overline{K(\overline{F})} \ \text{and} \ \exists x\in V \ x^{-1}A\in\U\}.\]
$I_{V}$ is close, $\overline{K(\overline{\F})}$ is compact and $\{I_{V}\}_{V\in\F}$ has the FIP, hence $\bigcap_{V\in\F} I_{V}\neq\emptyset$. Pick $\U$ in this intersection. Let $\beta\in\mu(\U)$. Then for all $V\in\F$, by construction, the set $J_{V}=\{x\in V\mid \beta \in x^{-1}A^{\ast}\}\neq\emptyset$. As the family $\{J_{V}\}_{V\in\F}$ has the FIP, we find $\gamma\in\bigcap_{V\in \F}J_{V}^{\ast}$. In particular, $\gamma\in\mu(\F)$ and $\gamma\beta^{\ast}\in A^{\ast\ast}$. As $\beta$ is minimal and $\gamma\in\mu(\F)$, so is $\gamma\beta^{\ast}$, and we conclude by Theorem \ref{T1}.
\end{proof}

\section{ $\mathcal{F}$-quasi-central sets and their dynamics}\label{four}

Quasi-central sets were first introduced in \cite{hindcen} to show that there are some sets belonging to any idempotent in the closure of the smallest two sided ideal of $\beta S$ (i.e., $K(\beta S)$) are combinatorially rich. In \cite{burns} they were  dynamically characterized. The second author of this paper studied quasi-central near zero in \cite{sou}. In \cite{ms} quasi-central sets  near an idempotent of a semitopological semigroup were discussed extensively.

In this present section we study 
quasi-central sets along a filter which will generalize all the above settings. To move forward, we need to state the following definition and theorem from \cite{john}.

 \begin{definition}\label{1.7}  Let $(X,\langle T_s\rangle _{s \in S})$ be a
dynamical system, $x$ and $y$ be two points in $X$, and $\mathcal{K}$ be a filter on $S$.
The pair $(x,y)$ is called jointly $\mathcal{K}$-recurrent if and only if for
every neighbourhood $U$ of $y$ we have that $\{ s \in S : T_s(x) \in U$ and
$T_s(y) \in U \} \in \mathcal{L}(\mathcal{K})$.
\end{definition}

\begin{theorem}\label{1.8}  Let $S$ be a semigroup, $\mathcal{K}$ be a filter on
$S$ such that $\bar{\mathcal{K}}$ is a compact subsemigroup of $\beta S$, and let
$A \subseteq S$. Then $A$ is a member of an idempotent in $\bar{\mathcal{K}}$ if
and only if there exists a dynamical system $(X,\langle T_s\rangle _{s \in S})$ with points
$x$ and $y$ in $X$ and there exists a neighbourhood $U$ of $y$ such that the
pair $(x,y)$ is jointly $\mathcal{K}$-recurrent and $A=\{ s \in S: T_s(x) \in U \}$.
\end{theorem}
\begin{proof}
 See \cite[Theorem 3.3]{john}.
\end{proof}
Theorem \ref{1.8} shows a beautiful relation between jointly $\mathcal{K}$-recurrent pairs and idempotent ultrafilters. Now we define quasi-central sets along a filter.

\begin{definition}\label{3.1}  Let $\mathcal{F}$ be a filter on $S$, and $C\subseteq S$.
Then $C$ is said to be $\mathcal{F}$-quasi-central if and only if there is an idempotent
$p$ in $cl K(\overline{\mathcal{F}})$ such that $C \in p$.
\end{definition}

It is well known that piecewise syndetic sets in $S$ can be characterized in terms of the closure of the smallest bilateral ideal of $\beta S$. \cite[Theorem 2.3]{oshu} generalizes this fact to piecewise $\mathcal{F}$-syndeticity.

\begin{theorem}\label{3.4} Let $\mathcal{F}$ be a filter on $S$ and  $A \subseteq S$.
Then $K(\overline{\mathcal{F}}) \cap cl_{\beta S_d }(A) \neq \emptyset$ if and only if $A$ is  piecewise $\mathcal{F}$-syndetic.
\end{theorem}

As an immediate consequence, we get the following characterization.

\begin{lemma}\label{3.5} Let $\mathcal{F}$ be a filter on $S$ and
let $$\mathcal{K} = \{A \subseteq S : S \setminus A \text{ is not  piecewise $\mathcal{F}$-syndetic }\}.$$
Then $\mathcal{K}$ is a filter on $S$ with $cl K(\overline{\mathcal{F}}) = \overline{\mathcal{K}}$,
which is a compact subsemigroup of $\beta S_d$.
\end{lemma}
\begin{proof} By the construction of $\mathcal{K}$ and Theorem \ref{3.4}, we have
$\mathcal{K} = \bigcap K(\overline{\mathcal{F}})$. Using Theorem 3.20(b) of \cite{hindalg}, we have
$\mathcal{K}$ is a filter and $\overline{\mathcal{K}} = clK(\overline{\mathcal{F}})$. By \cite[Theorem 2.15]{hindalg},
$clK(\overline{\mathcal{F}})$
is a right ideal of $\overline{\mathcal{F}}$, so in particular, $\overline{\mathcal{K}}$ is a compact subsemigroup of
$\overline{\mathcal{F}}$.
Therefore $cl K(\overline{\mathcal{F}}) = \overline{\mathcal{K}}$ is a compact subsemigroup of $\beta S$.\end{proof}

Let us now define 
 jointly intermittently $\mathcal{F}$-uniform recurrence which
will be helpful to give a dynamical characterization of quasi-central sets  along a filter.

\begin{definition}\label{3.3}  Let $(X, \langle T_s\rangle_{s\in S})$ be a dynamical system
and let $x,y \in X$. The pair $(x,y)$ is jointly intermittently $\mathcal{F}$-uniformly recurrent
(abbreviated as $JI\mathcal{F}UR$) if and only if for every neighbourhood $U$ of $y$,
the set $\{s \in S : T_s(x) \in U \text{ and } T_s(y)\in U\}$ is  piecewise $\mathcal{F}$-syndetic.
\end{definition}

Now  we are in the position to characterize  quasi-central sets dynamically along a filter  in terms of $JI\mathcal{F}UR$ pairs.

\begin{theorem}\label{3.6}   Let $\mathcal{F}$ be a filter on $S$ and let $A \subseteq S$.
The set $A$ is $\mathcal{F}$-quasi-central  if and only if there exists a dynamical system $(X,\langle T_s \rangle_{s \in S})$,
points $x$ and $y$ in $X$, and a neighbourhood $U$ of $y$ such that the
pair $(x,y)$ is $JI\mathcal{F}UR$ and $A = \{s \in S : T_s(x) \in U\}$.
\end{theorem}
\begin{proof} We shall prove this theorem using Theorem \ref{1.8}. Let $$\mathcal{K} = \{B \subseteq S : S\setminus B \text{ is not a  piecewise $\mathcal{F}$-syndetic set}\}.$$
Clearly, $\mathcal{L}(\mathcal{K}) = \{A \subseteq S : A \text{ is  piecewise $\mathcal{F}$-syndetic }\}$.
By Lemma  \ref{3.5}, we have $\mathcal{K}$ is a filter and $\overline{\mathcal{K}} = cl K(\overline{\mathcal{F}})$ which is a
compact subsemigroup of $\beta S$. Now we can apply Theorem \ref{1.8} to prove our required statement. \end{proof}

\section{Combinatorial characterization of large sets along a filter}\label{five}

In \cite{hindcen} Hindman, Maleki, and Strauss gave combinatorial characterizations of central sets and quasi-central sets. To characterize these large sets, syndetic sets, piecewise syndetic sets, and collectionwise piecewise syndetic sets played  significant roles.   
Motivated by this, in this section we want to study the combinatorial characterizations of large sets along a filter using the notions of $\mathcal{F}$-syndetic sets, piecewise $\mathcal{F}$-syndetic, and  collectionwise piecewise $\mathcal{F}$-syndetic. So, at first, we need to define the notion of collectionwise piecewise $\mathcal{F}$-syndeticity for further discussions.

\begin{definition}\label{new d4.1} Let $T$ be a closed subsemigroup of $\beta S$, $\mathcal{F}$ be a filter on $S$ such that $\overline{\mathcal{F}}=T$. A family $\mathcal A\subseteq \mathcal P(S)$ is collectionwise piecewise $\mathcal{F}$-syndetic if and only if for every $V\in \mathcal{F}$ there exist functions $G_V:\mathcal{P}_f(\mathcal A)\longrightarrow \mathcal{P}_f(V)$ and $\delta_V:\mathcal{P}_f(\mathcal A)\longrightarrow \mathcal{F}$ such that the family $$\{y^{-1}(G_V(\mathcal B))^{-1}(\cap \mathcal B)\cap V: V\in \mathcal{F}, y\in \delta_V(\mathcal B) \ \text{and} \ \mathcal B\in \mathcal{P}_f(\mathcal A)\}$$ has the finite intersection property.
\end{definition}

\begin{theorem}\label{new 4.2}
 Let $T$ be a closed subsemigroup of $\beta S$, $\mathcal{F}$ be a filter on $S$ such that $\overline{\mathcal{F}}=T$ and
 $\mathcal A\subseteq \mathcal P(S)$. Then there exists $p\in K(T)$ such that $\mathcal A\subseteq p$ if and only if $\mathcal A$ is collectionwise piecewise $\mathcal{F}$-syndetic.
\end{theorem}

\begin{proof} To prove the necessity, we first observe that for each $B\in \mathcal{P}_f(\mathcal A)$, $\cap \mathcal B\in p$. Let $L=Tp$ and $\mathcal G$ denote the filter on $S$ such that $\overline{\mathcal{G}}=L$. By Lemma \ref{new 2.4}, $\cap \mathcal B$ is $(\mathcal{F}, \mathcal{G})$-syndetic. Consequently, for every $V\in \mathcal F$, there is $G_V(\mathcal B)\in \mathcal{P}_f(V)$ such that $(G_V(\mathcal B))^{-1}(\cap \mathcal B)\in \mathcal{G}$. Since $L=Tp$, it follows that for every $V\in \mathcal{F}$, there exists $\delta_V(\mathcal B)\in \mathcal{F}$ such that $(\delta_V(\mathcal B))p\subseteq \overline{(G_V(\mathcal B))^{-1}(\cap \mathcal B)}$, and so $y^{-1}(G_V(\mathcal B))^{-1}(\cap \mathcal B)\in p$ for all $y\in \delta_V(\mathcal B))$. Hence the family
$$\{y^{-1}(G_V(\mathcal B))^{-1}(\cap \mathcal B)\cap V: V\in \mathcal{F}, y\in \delta_V(\mathcal B) \ \text{and} \ \mathcal B\in \mathcal{P}_f(\mathcal A)\}$$
has the finite intersection property.

To prove the sufficience, for each $V\in \mathcal{F}$ pick functions $G_V$ and $\delta_V$ as guaranteed by the assumption that $\mathcal A$ is collectionwise piecewise $\mathcal{F}$-syndetic. Given $V\in \mathcal{F}$, $F\in \mathcal{P}_f(S)$, pick $t(\mathcal B, F, V)\in V$ such that for every $C\in \mathcal{P}_f(\mathcal B)$ $(F\cap V)t(\mathcal B, F, V)\subseteq \bigcup_{x\in G_V(C)}x^{-1}(\cap C)$. For each $C\in \mathcal{P}_f(\mathcal A)$ and $y\in S$, let $$D(C, y)=\{t(\mathcal B, F, V)\mid \mathcal B\in \mathcal{P}_f(\mathcal A), F\in \mathcal{P}_f(S), y\in F, C\subseteq \mathcal B, V\in \mathcal{F}\}.$$ Then $\{D(C, y)\mid C\in \mathcal{P}_f(\mathcal A)$ and $y\in S\}\cup \mathcal F$ has the finite intersection property.

Indeed, given $\mathcal C_1, \mathcal C_2, \ldots, \mathcal C_n, y_1, y_2, \ldots, y_n$ and $\mathcal V_1, \mathcal V_2, \ldots, \mathcal V_n$, let $\mathcal B=\bigcup_{i=1}^n\mathcal C_i$, $F=\{y_1, y_2, \ldots, y_n\}$ and $\mathcal V=\bigcap_{i=1}^n\mathcal V_i$. Then $t(\mathcal B, F, V)\in \bigcap_{i=1}^n(D(\mathcal C_i, y_i)\cap V_i)$. So, pick $u\in T$ such that $\{D(\mathcal C, y): \mathcal C\in \mathcal{P}_f(\mathcal A)$ and $y\in S\}\subseteq u$. Now we claim that for each $\mathcal C\in \mathcal{P}_f(\mathcal A)$ and each $\mathcal V\in \mathcal F$, $Tu\subseteq \bigcup_{x\in G_V(\mathcal C)}\overline{x^{-1}(\cap \mathcal C)}$. To this end, let $q\in T$ and let $A=\bigcup_{x\in G_V(\mathcal C)}\overline{x^{-1}(\cap \mathcal C)}$. We claim that $\delta_V(\mathcal C)\subseteq \{y\in S: y^{-1}A\in u\}$, so that, since $\delta_V(\mathcal C)\in \mathcal F\subseteq q$, we have $A\in q+u$. Let $y\in \delta_V(\mathcal C)$. It suffices to show that $D(\mathcal C, y)\subseteq y^{-1}A$. So let $\mathcal B\in \mathcal{P}_f(\mathcal A)$ with $\mathcal C\subseteq \mathcal B$, let $F\in \mathcal{P}_f(S)$ with $y\in F$, let $\mathcal V\in \mathcal F$ be given. Then $y\in F\cap \delta_V(\mathcal C)$ so $yt(\mathcal B, F, V)\in A$ as required.

Pick a minimal left ideal $L$ of $T$ with $L\subseteq Tu$. Then $$L\subseteq \bigcap_{\mathcal C\in \mathcal{P}_f(\mathcal A)}\bigcap_{V\in \mathcal F}\bigcup_{x\in G_V(\mathcal C)}\overline{x^{-1}(\cap \mathcal C)}.$$ Pick $r\in L$. For each $\mathcal C\in \mathcal{P}_f(\mathcal A)$ and each $V\in \mathcal F$, pick $x(\mathcal C, V)\in G_V(\mathcal C)$ such that $(x(\mathcal C, V))^{-1}(\cap \mathcal C)\in r$. For each $\mathcal C\in \mathcal{P}_f(\mathcal A)$, let $\varepsilon(\mathcal C)=\{x(\mathcal B, V): \mathcal B\in \mathcal{P}_f(\mathcal A), \mathcal C\subseteq \mathcal B$ and $V\in  \mathcal F\}$. We claim that $\{\varepsilon(\mathcal C): \mathcal C\in \mathcal{P}_f(\mathcal A)\}\cup \mathcal F$ has the finite intersection property.

Indeed given $\mathcal C_1, \mathcal C_2, \ldots, \mathcal C_m$ and $ V_1,  V_2, \ldots,  V_m$ pick $V=\bigcap_{i=1}^m V_i$ and let $\mathcal B=\bigcup_{i=1}^m\mathcal C_i$. Then $x(\mathcal B, V)\in \bigcap_{i=1}^m\varepsilon(\mathcal C_i)\cap \bigcap_{i=1}^m V_i$. So, pick $w\in T$ such that $\{\varepsilon(\mathcal C): \mathcal C\in \mathcal{P}_f(\mathcal A)\}\subseteq w$. Let $p=wr$. Then $p\in L\subseteq K(T)$. To see that $\mathcal A\subseteq p$, let $A\in \mathcal A$. We show that $\varepsilon(\{A\})\subseteq \{x\in S: x^{-1}A\in r\}$, let $\mathcal C\in \mathcal{P}_f(\mathcal A)$ with $A\in \mathcal C$ and let $V\in \mathcal F$. Then $(x(\mathcal C, V))^{-1}(\cap \mathcal C)\in r$ and $(x(\mathcal C, V))^{-1}(\cap \mathcal C)\in (x(\mathcal C, V))^{-1}A$. \end{proof}

We recall the notion of tree below. We let $\omega=\{0, 1, 2, \ldots \}$ be the first transfinite ordinal number; we recall that in Von Neumann representation each ordinal can be identified with the set of its predecessors.


\begin{definition}\label{new d4.3} $\mathcal T$ is a tree in $A$ if and only if $\mathcal T$ is a set of functions and for each $f\in \mathcal T$, domain$(f)\in \omega$ and range$(f)\subseteq A$ and if domain$(f)=n>0$, then $f|_{n-1}\in \mathcal T$. $\mathcal T$ is a tree if and only if for some $A$, $\mathcal T$ is a tree in $A$.
\end{definition}

\begin{definition}\label{new d4.4} We fix the following notations.
 \begin{enumerate}
  \item [(a)] Let $f$ be a function with domain$(f)=n\in \omega$ and let $x$ be given. Then $f\frown x=f\cup \{(n, x)\}$.
  \item [(b)] Given a tree $\mathcal T$ and $f\in \mathcal T$, $B_f=B_f(\mathcal T)=\{x: f\frown x\in \mathcal T\}$
  \item [(c)] Let $S$ be semigroup and let $A\subseteq S$. Then $\mathcal T$ is a $\ast$-tree in $A$ if and only if $\mathcal T$ is a tree in $A$ and for all $f\in \mathcal T$ and all $x\in B_f$, $B_{f\frown x}\subseteq x^{-1}B_f$.
  \item [(d)] Let $S$ be semigroup and let $A\subseteq S$. Then $\mathcal T$ is a $FS$-tree in $A$ if and only if $\mathcal T$ is a tree in $A$ and for all $f\in \mathcal T$,  $$B_f=\left\{\prod_{t\in F}g(t): g\in \mathcal T, f\subseteq g, \ \text{and} \  \phi\neq F\subseteq \text{domain}(g)\setminus \text{domain}(f)\right\}.$$
 \end{enumerate}

\end{definition}

First, let us recall two results about $FS$-trees. The first is \cite[Lemma 3.6]{hindcen}.
\begin{lemma}\label{new 4.5}
 Let $S$ be semigroup and let $p$ be an idempotent in $\beta S$. If $A\in p$, then there is a $FS$-tree $\mathcal T$ in $A$ such that for each $f\in \mathcal T$, $B_f\in p$.
\end{lemma}

The second is \cite[Lemma 4.6]{hindult}.

\begin{lemma}\label{new 4.6}
 Any $FS$-tree is a $\ast$-tree.
\end{lemma}

We can now prove the equivalence of several trees properties when localizing along a filter.

\begin{theorem}\label{new 4.7}
Let $T$ be a closed subsemigroup of $\beta S$, $\mathcal{F}$ be a filter on $S$ such that $\overline{\mathcal{F}}=T$ and let $A\subseteq S$.\\ Statements (a), (b), (c), and (d) are equivalent and implied by statement (e). If $S$ is countable, all five statements are equivalents.
\begin{enumerate}
 \item [(a)] $A$ is $\mathcal{F}$-central

 \item [(b)] There is a $FS$-tree $\mathcal T$ in $A$ such that $\{B_f: f\in \mathcal T\}$ is collectionwise piecewise $\mathcal{F}$-syndetic.

 \item [(c)] There is a $\ast$-tree $\mathcal T$ in $A$ such that $\{B_f: f\in \mathcal T\}$ is collectionwise piecewise $\mathcal{F}$-syndetic.

 \item [(d)] There is a downward directed family $\langle C_F\rangle_{F\in I}$ of subsets of $A$ such that \begin{enumerate}
 \item [(i)] for all $F\in I$ and all $x\in C_F$, there is some $G\in I$ with $C_G\subseteq x^{-1}C_F$

 \item [(ii)] $\{C_F: F\in I\}$ is collectionwise piecewise $\mathcal{F}$-syndetic.

 \end{enumerate}

 \item [(e)] There is a decreasing sequence $\langle C_n\rangle_{n=1}^\infty$ of subsets of $A$ such that \begin{enumerate}
 \item [(i)] for all $n\in \mathbb{N}$ and all $x\in C_n$, there is some $m\in \mathbb{N}$ with $C_m\subseteq x^{-1}C_n$
 and
 \item [(ii)] $\{C_n: n\in \mathbb{N}\}$ is collectionwise piecewise $\mathcal{F}$-syndetic.
\end{enumerate}
\end{enumerate}
\end{theorem}

\begin{proof} (a) implies (b). By Lemma \ref{new 4.5}, pick a $FS$-tree $\mathcal T$ in $A$ such that for each $f\in \mathcal T$, $B_f\in p$. By Theorem \ref{new 4.2}, $\{B_f: f\in \mathcal T\}$ is collectionwise piecewise $\mathcal{F}$-syndetic.

That (b) implies (c) follows from Lemma \ref{new 4.6}.

(c) implies (d). Let $\mathcal T$ be given as guaranteed by (c). Let $I=\mathcal{P}_f(\mathcal T)$ and for $F\in I$, let $C_F=\bigcap_{f\in F}B_f$. Since $\{B_f: f\in T\}$ is collectionwise piecewise $\mathcal{F}$-syndetic, so is $\{C_F: F\in I\}$. Given $F\in I$ and $x\in C_F$, let $G=\{f\frown x: f\in F\}$. For each $f\in F$ we have $B_{f\frown x}\subseteq x^{-1}B_f$ by the definition of $\ast$-tree so $C_G=\bigcap_{f\in F}B_{f\frown x}\subseteq \bigcap_{f\in F}x^{-1}B_f=x^{-1}C_F$.

(d) implies (a). Let $M=\bigcap_{F\in I}\overline{C_F}$. We claim that $M$ is a subsemigroup of $\beta S$. To this end, let $p, q\in M$ and let $F\in I$. To see that $C_F\in pq$, we show that $C_F\subseteq \{x\in S: x^{-1}C_F\in q\}$. Let $x\in C_F$ and pick $G\in I$ such that $C_G\subseteq x^{-1}C_F$. Then $C_G\in q$ so $x^{-1}C_F\in q$.

By Theorem \ref{new 4.2} we have $M\cap K(T)\neq \phi$. Since $K$ is the union of all minimal left ideal of $T$ (see \cite[Theorem 1.3.11]{bergl}), pick a minimal left ideal $L$ of $K(T)$ with $M\cap L\neq \phi$. Then $M\cap L$ is a compact semigroup so by \cite[Corollary 2.10]{elli},
there is some $p=p\cdot p$ in $M\cap L$. Since each $C_F\subseteq A$, we have $p\in K(T)\cap \overline A$.

That (e) implies (d) is trivial.

Now assume that $S$ is countable. We show that (c) implies (e), so let $\mathcal T$ be as guaranteed by (c). Since $\mathcal T$ is countable, enumerate $\mathcal T$ as $\langle f_n\rangle_{n=1}^\infty$. For each $n\in \mathbb{N}$, let $C_n=\bigcap_{k=1}^nB_{f_k}$. Then $\{C_n: n\in \mathbb{N}\}$ is collectionwise piecewise $\mathcal{F}$-syndetic. Let $n\in \mathbb{N}$ be given and let $x\in C_n$. Then for each $k\in \mathbb{N}$, $B_{f_k\frown x}\subseteq x^{-1}B_{f_k}$. Pick $m\in \mathbb{N}$ such that $\{f_k\frown x: k\in \{1, 2, \ldots, n\}\}\subseteq \{f_k: k\in \{1, 2, \ldots, m\}\}$. Then $C_m\subseteq x^{-1}C_n$.
\end{proof}

\begin{theorem}\label{new 4.8}
 Let $T$ be a closed subsemigroup of $\beta S$, $\mathcal{F}$ be a filter on $S$ such that $\overline{\mathcal{F}}=T$ and let $A\subseteq S$.\\ Statements (a), (b), (c), and (d) are equivalent and implied by statement (e). If $S$ is countable, all five statements are equivalents.
 \begin{enumerate}
 \item [(a)] $A$ is $\mathcal{F}$-quasi-central.

 \item [(b)] There is a $FS$-tree $\mathcal T$ in $A$ such that for each $F\in \mathcal P_f(\mathcal T)$, $\bigcap_{f\in F}B_f$ is  piecewise $\mathcal{F}$-syndetic.

 \item [(c)] There is a $\ast$-tree $\mathcal T$ in $A$ such that for each $F\in \mathcal P_f(\mathcal T)$, $\bigcap_{f\in F}B_f$ is  piecewise $\mathcal{F}$-syndetic.

 \item [(d)] There is a downward directed family $\langle C_F\rangle_{F\in I}$ of subsets of $A$ such that \begin{enumerate}
 \item [(i)] for each $F\in I$ and each $x\in C_F$, there exists $G\in I$ with $C_G\subseteq x^{-1}C_F$ and

 \item [(ii)] for each  $F\in I$, $C_F$ is  piecewise $\mathcal{F}$-syndetic.

 \end{enumerate}

 \item [(e)] There is a decreasing sequence $\langle C_n\rangle_{n=1}^\infty$ of subsets of $A$ such that \begin{enumerate}
 \item [(i)] for each  $n\in \mathbb{N}$ and each $x\in C_n$, there exists $m\in \mathbb{N}$ with $C_m\subseteq x^{-1}C_n$
 and
 \item [(ii)] For each  $n\in \mathbb{N}$, $C_n$ is  piecewise $\mathcal{F}$-syndetic.
\end{enumerate}
\end{enumerate}
\end{theorem}

\begin{proof} The proof is same as Theorem \ref{new 4.7}. \end{proof}

\section{Minimal systems along filters}\label{six}
In  Section 2 of \cite{hindrec} Hindman, Strauss, and Zamboni presented some well known results about  
  $U(x)$ (see Definition \ref{1.4}(2)) that are true in an arbitrary dynamical system as well as the few simple results in
 $(\beta S,  \langle \lambda_s \rangle_{s \in S})$ such as (i) $U(x)=\beta S$ if $x$ is uniformly recurrent, (ii)
 for every $x\in X$, $U(x)$ is  left ideal of $\beta S$ containing $K(\beta S)$, (iii) $\bigcap_{x\in X}U(x)$ is
 a left as well as the right ideal of $\beta S$ (iv) $K(\beta S)$ is not
prime with some weak cancellation assumptions. These results were studied near an idempotent of a semitopological semigroup in \cite{ms}. In this present section, we shall establish these results along a filter, i.e., in more general settings.

\begin{definition}\label{d4.1.2}
 Let $(X,  \langle T_s \rangle_{s \in S})$ be a dynamical system and $x \in X$. Let $\mathcal{F}$ be a filter on $S$, then
 \begin{enumerate}
  \item $U_{\mathcal{F}}(x)=U_{{\mathcal{F}}_X}(x)=\{p\in \beta S: T_p(x)$ is $\mathcal{F}$-uniformly recurrent\}.
\item A subspace $Z$ of $X$ is called $\mathcal{F}$-invariant  if $T_p(Z)\subseteq Z$ for every $p\in \overline{\mathcal{F}}$.
 \end{enumerate}
\end{definition}
 \begin{lemma}\label{5.1}
 Let $(X,  \langle T_s \rangle_{s \in S})$ be a dynamical system and  $\mathcal{F}$ be a filter on $S$. Then the following are equivalent:-
 \begin{enumerate}
  \item $x$ is $\mathcal{F}$-uniformly recurrent.
  \item There exists $q \in K(\overline{\mathcal{F}})$ such that $T_q(x)=x$.
  \item There exists $y \in X$ and $q \in K(\overline{\mathcal{F}})$ such that $T_q(y)=x$.
 \end{enumerate}

\end{lemma}
\begin{proof} The implications of (1) to (2) and (2) to (3) simply follow from Lemma \ref{2.9}. We shall show (3) implies
(1) and that will provide the equivalence of all three statements. To this end, let $L$ be a minimal left ideal of
$K(\overline{\mathcal{F}})$ such that $q \in L$ and $u$ be the identity of the group $L\cap (qK(\overline{\mathcal{F}}))$. Since $uq=q$, it follows
that $T_u(x)=T_u(T_q(y))=T_{uq}(y)=T_q(y)=x$. Again by Lemma \ref{2.9}, $x$ is $\mathcal{F}$-uniformly recurrent.\end{proof}

\begin{corollary}\label{5.3} Let $(X,  \langle T_s \rangle_{s \in S})$ be a dynamical system and $x \in X$. Let $\mathcal{F}$ be a filter on $S$, then
\begin{enumerate}
 \item if $x$ is $\mathcal{F}$-uniformly recurrent  then  $\overline{\mathcal{F}}\subseteq U_{\mathcal{F}}(x)$,
 \item for each $x \in X$, $K(\overline{\mathcal{F}})\subseteq U_{\mathcal{F}}(x)$,
 \item for each $x \in X$, $U_{\mathcal{F}}(x)\cap \overline{\mathcal{F}}$ is a left ideal of $\overline{\mathcal{F}}$,
 \item $(\bigcap_{x \in X}U_{\mathcal{F}}(x))\bigcap \overline{\mathcal{F}}$ is a two-sided ideal of $\overline{\mathcal{F}}$.
 \end{enumerate}
\end{corollary}
\begin{proof} (1) Suppose that $x$ is $\mathcal{F}$-uniformly recurrent. Then by Lemma \ref{5.1}, $T_u(x)=x$ for some
$u\in K(\overline{\mathcal{F}})$. Thus for every $v\in \overline{\mathcal{F}}$, $T_v(x)=T_v(T_u(x))=T_{vu}(x)$. Now since $uv\in K(\overline{\mathcal{F}})$,
by Lemma \ref{2.9} $T_v(x)$ is $\mathcal{F}$-uniformly recurrent  and thus, $v\in U_{\mathcal{F}}(x)$.
Therefore $\overline{\mathcal{F}}\subseteq U_{\mathcal{F}}(x)$.\\
(2) This is immediate from  Lemma \ref{2.9}.\\
(3) Let $x\in X$, $p\in U_{\mathcal{F}}(x)\cap \overline{\mathcal{F}}$ and $r\in \overline{\mathcal{F}}$. By Lemma \ref{2.9} pick $q\in K(\overline{\mathcal{F}})$ such
that $T_q(T_p(x))=T_p(x)$. Then $T_{rp}(x)=T_r(T_q(T_p(x)))=T_{rqp}(x)$. Now $rqp\in  K(\overline{\mathcal{F}})$.
So by Lemma \ref{2.9}, $T_{rp}(x)$ is $\mathcal{F}$-uniformly recurrent and hence $rp\in U_{\mathcal{F}}(x) \cap \overline{\mathcal{F}}$.
Therefore $U_{\mathcal{F}}(x)\cap \overline{\mathcal{F}}$ is a left ideal of $\overline{\mathcal{F}}$.\\
(4) By (2) $(\bigcap_{x \in X}U_{\mathcal{F}}(x))\bigcap \overline{\mathcal{F}}$ is nonempty. So by (3),
$(\bigcap_{x \in X}U_{\mathcal{F}}(x))\bigcap \overline{\mathcal{F}}$
is a left ideal of $\overline{\mathcal{F}}$. So it is enough to show that $(\bigcap_{x \in X}U_{\mathcal{F}}(x))\bigcap \overline{\mathcal{F}}$ is a right ideal
of $\overline{\mathcal{F}}$. To this end, let $p\in (\bigcap_{x \in X}U_{\mathcal{F}}(x))\bigcap \overline{\mathcal{F}}$ and
$q\in \overline{\mathcal{F}}$. Suppose $y\in X$
then $p\in U_{\mathcal{F}}(T_q(y))$. Thus $T_{pq}(y)$ is $\mathcal{F}$-uniformly recurrent  and so $pq\in U_{\mathcal{F}}(y)$.\end{proof}

The proofs of the following Lemma and the next Theorem follow closely the arguments of Hindman, Strauss, and Zamboni
in \cite{hindrec}.

\begin{lemma}\label{5.4}
 Let $(X,  \langle T_s \rangle_{s \in S})$ be a dynamical system and $L$ be a minimal left ideal of $\overline{\mathcal{F}}$.
 \begin{enumerate}
  \item A subspace $Y$ of $X$ is minimal among all closed and $\mathcal{F}$-invariant subspaces of $X$  if and only if there
 is some $x\in X$ such that $Y=\{T_p(x): p\in L\}$.
 \item  Let $Y$ be a subspace of $X$ which is minimal among all closed and $\mathcal{F}$-invariant subspaces of $X$.
 Then every element of $Y$ is $\mathcal{F}$-uniformly recurrent.
 \item  If $x\in X$ is $\mathcal{F}$-uniformly recurrent  and $Y=\{T_p(x): p\in \overline{\mathcal{F}}\}$, then $Y$ is minimal among all
 closed and $\mathcal{F}$-invariant subspaces of $X$.
 \item  If $x\in X$ is $\mathcal{F}$-uniformly recurrent  then $T_p(x)$ is $\mathcal{F}$-uniformly recurrent for every $p\in \overline{\mathcal{F}}$.
\end{enumerate}
\end{lemma}

\begin{proof} (1) Suppose that $Y$ is minimal among all closed and $\mathcal{F}$-invariant subspaces of $X$. Pick $x\in Y$
 and let $Z=\{T_p(x): p\in L\}$. We show that $Z$ is a closed and $\mathcal{F}$-invariant subspace of $Y$ and this is equal
 to $Y$. If $p\in L$ and $q\in \overline{\mathcal{F}}$, then $T_q(T_p(x))=T_{qp}(x)$ and $qp\in L$. So $Z$ is $\mathcal{F}$-invariant  and
 obviously $Z \subseteq Y$. To prove that $Z$ is closed, it is enough to show that any net in $Z$ has a cluster point in $Z$.

 To this end, let $\langle p_{\alpha}\rangle_{\alpha\in D}$ be a net in $L$ and pick a cluster point $p$ in $L$ of
$\langle p_{\alpha}\rangle_{\alpha\in D}$.
 Then $T_p(x)$ is a cluster point of $\langle T_{p_{\alpha}}(x)\rangle_{\alpha\in D}$.

 Conversely, let $x\in X$ and $Y=\{T_p(x): p\in L\}$. Then  $Y$ is $\mathcal{F}$-invariant  and is closed as above.
 We now show that $Y$ is minimal among all closed $\mathcal{F}$-invariant subspaces of $X$. Suppose that $Z$ is a subspace of
 $Y$ which is closed and $\mathcal{F}$-invariant. We shall show that $Y \subseteq Z$. So let $y\in Y$ and pick $z\in Z$.
 Then $y=T_p(x)$ and $z=T_q(x)$ for some $p$ and $q$ in $L$. Since $Lq=L$, there exists $r\in L$ such that $rq=p$.
It follows that $T_r(z)=T_r(T_q(x))=T_{rq}(x)=T_p(x)=y$ and thus $y\in Z$ as required.

(2) Let $Y$ be a subspace of $X$, which  is minimal among all closed and $\mathcal{F}$-invariant subspaces of $X$  and $x\in Y$. Pick $y\in X$ such that
 $Y=\{T_p(x): p\in L\}$. Pick $p\in L$ such that $x=T_p(y)$. By Lemma \ref{2.9}, $x$ is $\mathcal{F}$-uniformly recurrent.

 (3) Let $x\in X$ be $\mathcal{F}$-uniformly recurrent  and $Y=\{T_p(x): p\in \overline{\mathcal{F}}\}$. By Lemma \ref{2.9}, pick $q\in L$
 such that $T_q(x)=x$. By (1), it suffices to show that $Y=\{T_p(x): p\in L\}$. To prove this, let $y\in Y$ and pick
 $p\in \overline{\mathcal{F}}$ such that $y=T_p(x)$. Then $y=T_p(T_q(x))=T_{pq}(x)$ and $pq\in L$ as required.

 (4) Let $x\in X$ be $\mathcal{F}$-uniformly recurrent  and $Y=\{T_p(x): p\in \overline{\mathcal{F}}\}$. By (3) $Y$ is minimal among all
 closed and $\mathcal{F}$-invariant subspaces of $X$ so (2) applies. \end{proof}

 \begin{theorem}\label{5.5}  Let $\mathcal{F}$ be a filter on $S$ and
 $x\in \overline{\mathcal{F}}$. Statements (a) and (b) are equivalent and imply (c). If $\overline{\mathcal{F}}$ has a left cancelable element, all
 three are equivalent.
 \begin{enumerate}
  \item [(a)] $x\in K(\overline{\mathcal{F}})$.
\item [(b)] $x\in X$ is $\mathcal{F}$-uniformly recurrent  in the dynamical system $(\beta S,  \langle \lambda_s \rangle_{s \in S})$.
\item [(c)] $\overline{\mathcal{F}}x$ is a minimal left ideal of $\overline{\mathcal{F}}$.
 \end{enumerate}
 \end{theorem}
\begin{proof} (a) implies (b). Let $x\in K(\overline{\mathcal{F}})$ and let $u$ be the identity of the group in $K(\overline{\mathcal{F}})$ to
 which $x$ belongs. Then $x=\lambda_u(x)$ so by Lemma \ref{5.1}, $x$ is $\mathcal{F}$-uniformly recurrent in the dynamical
 system $(\beta S,  \langle \lambda_s \rangle_{s \in S})$.

 (b) implies (a). Let $x$ be $\mathcal{F}$-uniformly recurrent  in the dynamical system $(\beta S,  \langle \lambda_s \rangle_{s \in S})$.
By Lemma \ref{5.1}, there exists $q\in K(\overline{\mathcal{F}})$ such that $\lambda_q(x)=x$. Then $x=qx\in K(\overline{\mathcal{F}})$.

 (a) implies (c). Assume that $x\in K(\overline{\mathcal{F}})$ and pick the minimal left ideal $L$ of $\overline{\mathcal{F}}$ such that $x\in L$.
Then $\overline{\mathcal{F}}x$ is a left ideal of $\overline{\mathcal{F}}$ contained in $L$. So $L=\overline{\mathcal{F}}x$. Now assume that $\overline{\mathcal{F}}$ has a
left cancelable element $z$ and $\overline{\mathcal{F}}x$ is a minimal left ideal of $\overline{\mathcal{F}}$. Pick an idempotent $u\in \overline{\mathcal{F}}x$.
Then $zx\in \overline{\mathcal{F}}x$. So by \cite[Lemma 1.30]{hindalg}, $zx=zxu$ and therefore
$x=xu\in \overline{\mathcal{F}}x\subseteq K(\overline{\mathcal{F}})$. \end{proof}

 \begin{corollary}\label{5.6} Let $\mathcal{F}$ be a filter on $S$ and $x\in K(\overline{\mathcal{F}})$.
Then $\overline{\mathcal{F}}\subseteq U_{\mathcal{F}}(x)$ with respect to the dynamical system $(\beta S,  \langle \lambda_s \rangle_{s \in S})$.
\end{corollary}
\begin{proof} By Theorem \ref{5.5}, $x$ is $\mathcal{F}$-uniformly recurrent, so by Lemma \ref{5.4},
$\overline{\mathcal{F}}\subseteq U_{\mathcal{F}}(x)$. \end{proof}

\begin{corollary}\label{5.7} Let $\mathcal{F}$ be a filter on $S$  and $p, q\in \overline{\mathcal{F}}$. Statements
(a) and (b) are equivalent and imply (c). If $\overline{\mathcal{F}}$ has a left cancelable element, all three are equivalent.
\begin{enumerate}
 \item [(a)] $qp\in K(\overline{\mathcal{F}})$.
 \item [(b)] $q\in U_{\mathcal{F}}(p)$ with respect to the dynamical system $(\beta S,  \langle \lambda_s \rangle_{s \in S})$.
 \item [(c)] $\overline{\mathcal{F}}qp$ is a minimal left ideal of $\overline{\mathcal{F}}$.
 \end{enumerate}
 \end{corollary}

\begin{proof} We have that $q\in U_{\mathcal{F}}(p)$ if and only if $\lambda_q(p)$ is $\mathcal{F}$-uniformly recurrent and
 $\lambda_q(p)=qp$, so Theorem \ref{5.5} applies. \end{proof}

 \begin{corollary}\label{5.8}  Let $\mathcal{F}$ be a filter on $S$. The following are equivalent.
 \begin{enumerate}
  \item [(a)] There exists $p\in \overline{\mathcal{F}}\setminus K(\overline{\mathcal{F}})$ such that $K(\overline{\mathcal{F}})\subset U_{\mathcal{F}}(p)$ with respect to the dynamical
system $(\beta S,  \langle \lambda_s \rangle_{s \in S})$.
 \item [(b)] $K(\overline{\mathcal{F}})$ is not prime.
\end{enumerate}
\end{corollary}
\begin{proof} The proof is an immediate consequence of Corollary \ref{5.7}. \end{proof}

\section{Partition regularity along filters}\label{seven}

As an example of application of the notions developed above, we discuss here some results about the partition regularity of equations along filters. One of the major problems in Ramsey theory regards the so-called partition regularity of equations (see \cite{advances} for a general introduction to the topic). In what follows, we let $T\in\{\N,\Z,\Q,\R\}$ and $S\subseteq T$. We will use the following notation: for $\F$ filter on $S$, we let 
\[\overline{\F}_{T}:=\{\U\in\beta T\mid \F\subseteq\U\}.\]

Notice that, if we let $\F( T):=\{A\in T\mid \exists B\in\F \ B\subseteq A\}$, then $\overline{F}_{T}=\overline{\F( T)}$.

\begin{definition} Let $S\subseteq T$, let $\mathbb{K}=\R$ if $T=\R$, $\mathbb{K}=\Q$ otherwise. Let $P_{1}\left(x_{1},\dots,x_{n} \right)$,$\dots$,$P_{m}\left(x_{1},\dots,x_{n} \right)\in  \mathbb{K}\left[x_{1},\dots,x_{n} \right]$. Let 
\[\sigma\left(x_{1},\dots,x_{n} \right)=\begin{cases}  P_{1}\left(x_{1},\dots,x_{n} \right),\\ \hspace{1.2cm}\vdots \\ P_{m}\left(x_{1},\dots,x_{n} \right). \end{cases}\]

We say that the system of equations $\sigma\left(x_{1},\dots,x_{n} \right)=(0,\dots,0)$ is\footnote{From now on, we will simply write $\sigma\left(x_{1},\dots,x_{n} \right)=0$ to simplify the notation.} partition regular on $S$ if it has a monochromatic solution in every finite coloring of $S\setminus\{0\}$, namely if for every natural number $r$, for every partition $S=\bigcup\limits_{i=1}^{r}A_{i}$, there is an index $j\leq r$ and numbers $a_{1},\dots,a_{n}\in A_{j}$ such that $\forall j\in\{1,\dots,m\} \ P_{j}\left(a_{1},\dots,a_{n} \right)=0$.

\end{definition}

The notion of partition regularity near a filter can be introduced as follows:

\begin{definition} Let $S\subseteq T$ and let $\F$ be a filter on $S$. Let $P_{1}\left(x_{1},\dots,x_{n} \right)$, $\dots$, $P_{m}\left(x_{1},\dots,x_{n} \right)\in  T\left[x_{1},\dots,x_{n} \right]$. Let 
\[\sigma\left(x_{1},\dots,x_{n} \right)=\begin{cases}  P_{1}\left(x_{1},\dots,x_{n} \right),\\ \hspace{1.2cm}\vdots \\ P_{m}\left(x_{1},\dots,x_{n} \right). \end{cases}\]

We say that the system of equations $\sigma\left(x_{1},\dots,x_{n} \right)=0$ is $\F$-partition regular on $S$ if for all $V\in\F$, for all finite partitions $S=A_{1}\cup\dots\cup A_{k}$ there exists $j\leq k$ and $a_{1},\dots,a_{n}\in A_{j}\cap V$ such that $\sigma\left(a_{1},\dots,a_{n} \right)=0$.
\end{definition}

In \cite{original near 0}, the second and third authors of this paper started the study of the partition regularity of equations in the case where $T=\R$, $S$ is an $HL$-semigroup and $\F=\{(0,\varepsilon)\cap S\mid \varepsilon\in  \R^{+}\}$ (see \cite{original near 0}). These results where then extended by the first author in \cite{mine near 0}; the methods used in \cite{mine near 0} actually use two generic properties of $\F$ and can, as such, be generalized, which is what we aim to do in this Section.

It is well known that partition regularity problems can be rephrased in terms of ultrafilters. As for $\F$-partition regularity, the following characterization (whose proof we omit) holds:

\begin{proposition}\label{ultra part F} Let $S\subseteq T$ and let $\F$ be a filter on $S$. Let $P_{1}\left(x_{1},\dots,x_{n} \right)$, $\dots$, $P_{m}\left(x_{1},\dots,x_{n} \right)\in T\left[x_{1},\dots,x_{n} \right]$. Let
\[\sigma\left(x_{1},\dots,x_{n} \right)=\begin{cases}  P_{1}\left(x_{1},\dots,x_{n} \right),\\ \hspace{1.2cm}\vdots \\ P_{m}\left(x_{1},\dots,x_{n} \right). \end{cases}\]
The system of equations $\sigma\left(x_{1},\dots,x_{n} \right)=0$ is $\F$-partition regular if and only there exists an ultrafilter $\U\in\overline{\F}$ such that $\forall A\in\U \ \exists a_{1},\dots,a_{n}\in A \ \sigma\left(a_{1},\dots,a_{n} \right)=0$.\end{proposition}

\begin{definition} Under the conditions of Proposition \ref{ultra part F}, we say that $\U$ witnesses the $\F$-partition regularity of the system $\sigma\left(x_{1},\dots,x_{n} \right)=0$, and we call it a $\iota_{\sigma}$-ultrafilter. \end{definition}

We now recall two results (see e.g. \cite{advances} for the proofs\footnote{In \cite{advances}, the proofs are done for $T=\N$, but the same exact proof would work for any $T\in\{\N,\Z,\Q,\R\}$.}) that we will be used in the following.

\begin{theorem}\label{bil id} Let $P_{1}\left(x_{1},\dots,x_{n} \right),\dots,P_{m}\left(x_{1},\dots,x_{n} \right)\in  T\left[x_{1},\dots,x_{n} \right]$ be homogeneous. Let
\[\sigma\left(x_{1},\dots,x_{n} \right)=\begin{cases}  P_{1}\left(x_{1},\dots,x_{n} \right),\\ \hspace{1.2cm}\vdots \\ P_{m}\left(x_{1},\dots,x_{n} \right). \end{cases}\]
Assume that the system of equations $\sigma\left(x_{1},\dots,x_{n} \right)=0$ is partition regular on $S$. Then the set
\[I_{\sigma}=\{\U\in\beta S\mid \U \ \text{is a} \ \iota_{\sigma}\text{-ultrafilter}\}\] is a closed bilateral ideal in $\left(\beta  S,\odot \right)$. In particular, every ultrafilter in $\overline{K(\beta S,\odot)}$ witnesses the partition regularity of all homogeneous partition regular systems on $S$.\end{theorem} 

Two immediate consequences of Theorem \ref{bil id} are the following:

\begin{corollary}\label{easy one} Let $\F$ be a filter on $S\subseteq T$. If every set $A\in\F$ is piecewise syndetic in $( S,\cdot)$, then all homogeneous partition regular systems on $S$ are also $\F$-partition regular. \end{corollary}

\begin{proof} By our hypothesis, there exists $\U\in \overline{K(\beta S,\odot)}$ that extends $\F$, and we conclude by Theorem \ref{bil id}. \end{proof}

For example, let $\F=\{A\subseteq \N\mid \exists n\in\N \{m\in\N\mid n|m\}\subseteq A\}$. Then every set in $A$ is piecewise syndetic in $(\N,\cdot)$, so every partition regular system is also $\F$-partition regular. 

\begin{corollary}\label{easy two} Let $\F$ be a filter on $S$ such that $\overline{\F}$ is a left or a right ideal in $\beta S$. Then an homogeneous system is partition regular on $S$ if and only if it is $\F$-partition regular. \end{corollary}

\begin{proof} Any $\F$-partition regular system is trivially partition regular. Conversely, assume that $\overline{\F}$ is a left ideal (the proof is similar when $\overline{\F}$ is a right ideal). Let $\sigma$ be an homogeneous partition regular system. If $\U$ is a witness of the partition regularity of $\sigma$ and $\V\in\F$ then $\U\otimes\V\in\overline{\F}$ is a witness of the $\F$-partition regularity of $\sigma$ by Theorem \ref{bil id}. \end{proof}

For example, from Corollary \ref{easy two} it follows that all homogeneous partition regular systems on $\R$ are also $\F$-partition regular for $\F=\{(0,\varepsilon)\mid \varepsilon >0\}$, as well as for $\F=\{(r,+\infty)\mid r>0\}$, as $\overline{\F}$ is a left ideal in $\beta \R$ in both these cases.

The second result, which is just a reformulation of \cite[Lemma 2.1]{advances}, allows us to mix different partition regular systems to produce new ones.

\begin{lemma}\label{lemmasystem}Let $S\subseteq T$. Let $P_{1}\left(x_{1},\dots,x_{n} \right)$,$\dots$,$P_{m}\left(x_{1},\dots,x_{n} \right)\in  T\left[x_{1},\dots,x_{n} \right]$, $Q_{1}\left(y_{1},\dots,y_{l} \right)$, $\dots$, $Q_{t}\left(y_{1},\dots,y_{l} \right)\in  T\left[y_{1},\dots,y_{l} \right]$. Let $\U\in\beta S$ be a witness of the partition regularity of the systems of equations $\sigma_{1}\left(x_{1},\dots,x_{n} \right)=0$, $\sigma_{2}\left(y_{1},\dots,y_{l} \right)=0$, where 
\[\sigma_{1}\left(x_{1},\dots,x_{n} \right)=\begin{cases}  P_{1}\left(x_{1},\dots,x_{n} \right),\\ \hspace{1.2cm}\vdots \\ P_{m}\left(x_{1},\dots,x_{n} \right) \end{cases}\]
and
\[\sigma_{2}\left(y_{1},\dots,y_{l} \right)=\begin{cases}  Q_{1}\left(y_{1},\dots,y_{l} \right),\\ \hspace{1.2cm}\vdots \\ Q_{t}\left(y_{1},\dots,y_{l} \right). \end{cases}\]
Then $\U$ witnesses also the partition regularity of $\sigma_{3}\left(x_{1},\dots,x_{n},y_{1},\dots,y_{l} \right)=0$, where
\[\sigma_{3}\left(x_{1},\dots,x_{n},y_{1},\dots,y_{l} \right)=\begin{cases}  P_{1}\left(x_{1},\dots,x_{n} \right),\\ \hspace{1.2cm}\vdots \\ P_{m}\left(x_{1},\dots,x_{n} \right), \\Q_{1}\left(y_{1},\dots,y_{l} \right),\\ \hspace{1.2cm}\vdots \\ Q_{t}\left(y_{1},\dots,y_{l} \right),\\ x_{1}-y_{1}.\end{cases}\]
\end{lemma}

Lemma \ref{lemmasystem} is useful to work with ultrafilters with multiple structures. For example, as first shown in \cite{integers}, if $\U$ is a multiplicatively idempotent ultrafilter in $\overline{K(\beta T,\odot)}$ then $\U$ witnesses the partition regularity of all equations of the following form
\[\sum\limits_{i=1}^{n} c_{i}x_{i}Q_{F_{i}}\left(y_{1},\dots,y_{m} \right)=0,\]
where $n\geq 2$ is a natural number, $R\left(x_{1},\dots,x_{n} \right)=\sum\limits_{i=1}^{n} c_{i}x_{i}\in T\left[x_{1},\dots,x_{n} \right]$ is partition regular on $ T$, $m$ is a positive natural number and, for every $i\leq n$, $F_{i}\subseteq\{1,\dots,m\}$ and $Q_{F_{i}}:=\prod_{j\in F_{i}} y_{j}$ (if $F_{i}=\emptyset$, we let $Q_{\emptyset}=1$). 

In analogy with what was done in \cite{mine near 0} for the partition regularity near $0$, we can prove the following general result about $\F$-partition regularity.

\begin{theorem}\label{main} Let $S\subseteq T$ and $\F$ be a filter on $S$. Assume that $\overline{\F}_{T}\cap \overline{K(\beta T,\odot)}$ contains a multiplicative idempotent. Let $\mathcal{C}_{\F}$ be the set of polynomial systems that are $\F$-partition regular. Then $\mathcal{C}_{\F}$ includes:
\begin{enumerate} 
\item all partition regular homogeneous systems on $ T$;
\item all equations of the form 
\[P\left(x_{1},\dots,x_{n},y_{1},\dots,y_{m} \right)=\sum\limits_{i=1}^{n} a_{i}x_{i}Q_{F_{i}}\left(y_{1},\dots,y_{m} \right)\] where $\sum\limits_{i=1}^{n} a_{i}x_{i}\in T\left[x_{1},\dots,x_{n} \right]$ is partition regular on $T$ and $F_{1},\dots,F_{n}\subseteq\{1,\dots,m\}$.
\end{enumerate}
Moreover, if
\[\sigma_{1}\left(x_{1},\dots,x_{n} \right)=\begin{cases}  P_{1}\left(x_{1},\dots,x_{n} \right),\\ \hspace{1.2cm}\vdots \\ P_{m}\left(x_{1},\dots,x_{n} \right) \end{cases}\]
and
\[\sigma_{2}\left(y_{1},\dots,y_{l} \right)=\begin{cases}  Q_{1}\left(y_{1},\dots,y_{l} \right),\\ \hspace{1.2cm}\vdots \\ Q_{t}\left(y_{1},\dots,y_{l} \right) \end{cases}\]
belong to $\mathcal{C}_{\F}$, then also
\[\sigma_{3}\left(x_{1},\dots,x_{n},y_{1},\dots,y_{l} \right)=\begin{cases}  P_{1}\left(x_{1},\dots,x_{n} \right),\\ \hspace{1.2cm}\vdots \\ P_{m}\left(x_{1},\dots,x_{n} \right), \\Q_{1}\left(y_{1},\dots,y_{l} \right),\\ \hspace{1.2cm}\vdots \\ Q_{t}\left(y_{1},\dots,y_{l} \right),\\ x_{1}-y_{1}\end{cases}\]
belongs to $\mathcal{C}_{\F}$.

\end{theorem}

\begin{proof} Let $\U$ be a multiplicative idempotent in $\overline{\F}_{T}\cap \overline{K(\beta T,\odot)}$. We just have to observe that, by Corollary \ref{easy two}, $\U$ is a witness of all $\F$-partition regular homogeneous systems, which are precisely all partition regular homogeneous systems, whilst the partition regularity of equations of the form (2) has been discussed before the Theorem. The closure with respect to composition follows by Lemma \ref{lemmasystem}. As $\F\subseteq\U$, we conclude by Proposition \ref{ultra part F}. \end{proof}

Notice that the request that $\overline{\F}_{T}\cap \overline{K(\beta T,\odot)}$ contains an idempotent ultrafilter is always true.


\end{document}